\documentclass[12pt, a4paper, leqno]{amsart}

\usepackage{amsmath} 
\usepackage{amsfonts}
\usepackage{amssymb}
\usepackage{txfonts}   
\usepackage{amsfonts}  
\usepackage{comment}
\usepackage{ae}   
\usepackage{graphicx}   
\usepackage{pst-all} 
 
\newcommand{\R}{\varmathbb{R}}
\newcommand{\Rn}{\varmathbb{R}^n}

\def\b{\qopname\relax o{b}}
\def\spt{\qopname\relax o{spt}}

\usepackage{amsthm}

\swapnumbers
\theoremstyle{plain}
\newtheorem{theorem}[equation]{Theorem}
\newtheorem{lemma}[equation]{Lemma}

\newtheorem{corollary}[equation]{Corollary}

\theoremstyle{definition}
\newtheorem{definition}[equation]{Definition}

\newtheorem{example}[equation]{Example}

\theoremstyle{remark}
\newtheorem{remark}[equation]{Remark}

\newtheorem{proof-CSTI}[equation]{Proof for the capacitary strong type inequality}
\newtheorem{proof-Riesz}[equation]{Proof for the equivalence of the capacities}
\numberwithin{equation}{section}

\date{\today}
\begin{document}

\title{On capacitary strong type inequalities for Orlicz-Sobolev functions}

\author{Ritva Hurri-Syrj\"anen}
\author{Jani Joensuu}
\address{Ritva Hurri-Syrj\"anen \newline University of Helsinki, Department of Mathematics and Statistics, 
Gustaf H\"allst\"omin katu 2 $\b$ , FI-00014 University of Helsinki, Finland}
\email{ritva.hurri-syrjanen@helsinki.fi}
\address{Jani Joensuu \newline Aalto University, Department of Mathematics and Systems Analysis, P. O. Box 11100, FI-00076 Aalto, Finland}
\email{jani.joensuu@aalto.fi}



\begin{abstract}
We prove capacitary strong type inequalities for functions belonging to Orlicz-Sobolev spaces. 
As an application we consider capacitary averages and their limits.
\end{abstract}

\maketitle
\markboth{\textsc{On capacitary strong type inequalities}}
{\textsc{Ritva Hurri-Syrj\"anen and Jani Joensuu}}

\section{Introduction} \label{1}

Let $\Phi$ be a Young function with some restrictions and let
$\Psi$ be an increasing $C^1$-function with some restrictions, both defined on the positive real line.
We study the capacitary strong type inequalities
\begin{equation} \label{CSTI}
\int_{0}^{\infty} C_{\Phi} \bigg(\bigg\{x\in B^n(0,R): |u(x)|>t\bigg\}\bigg)
d \Psi(t) \leq K \int \Phi (|\nabla u(x)|) dx
\end{equation}
for all functions $u$ belonging to the Orlicz-Sobolev space 
$W_{0}^{1, \Phi}(B^n(0,R))$, $n\geq 2$; here
$K$ is a constant
independent of $u$.

When $\Psi (t)=t^p,$ $t>0,$ $p\geq 1,$
capacitary strong type inequalities 
have been studied thoroughly
by Vladimir Maz'ya in \cite[Sections 11.1--11.4]{Maz} and
by David R. Adams and Lars Inge Hedberg in 
\cite[Section 7]{AH}.

The cases when 
$\Phi (t)=t^n(\log (e+t))^{\alpha}$
and $\Psi (t)=t^n(\log (e+\frac{1}{t}))^{-\alpha},$  $t>0$, $0\le \alpha \le n-1$,
were considered by
Adams and the first author in \cite{AHS2} where the capacitary strong type inequality \eqref{CSTI} is proved for functions whose
derivatives belong to $L^n(\log (e+L))^{\alpha}$, 
\cite[Theorem 1.9]{AHS2}.
We generalize these results for a larger class of functions $\Phi$ and $\Psi$.
Our main theorem is 

\begin{theorem} \label{Main}
Let $\Phi$ be a Young function on $[0,\infty )$
and let $\Psi$ be an increasing $C^1$-function on $(0,\infty )$ such that
there are functions
$f$ and $\varphi :[0,\infty) \to [0,\infty)$ satisfying the doubling condition and
a function
$\psi:(0,\infty)\to(0,\infty)$   so that
$\Phi(t)=f(t) \varphi(t)$ and $\Psi(t)=f(t)\psi(t)$ and
the inequalities
\begin{equation} \label{EP}
f(s) \cdot f(t) \leq C f(st)
\end{equation}
and  
\begin{equation} \label{EP2}
\varphi(s) \cdot {\psi}(t)\leq C \varphi(st)
\end{equation}
hold for all $s,t  \in (0,\infty)$
with some fixed constant $C$.
Let $R>0$ be given.
Then, there is a constant $K$
such that the capacitary strong type inequality
\begin{equation} \label{CI}
\int_{0}^{\infty} C_{\Phi} (\{x\in B^n(0,R): |u(x)|>t\})d \Psi(t) \leq K \int \Phi (|\nabla u(x)|) dx
\end{equation}
holds for all functions $u$ belonging to the Orlicz-Sobolev space 
$W_{0}^{1, \Phi}(B^n(0,R))$, $n\geq 2$; here
$K$ is 
independent of $u$.
\end{theorem}

\begin{remark}
If $\Phi(t)=f (t)\varphi(t)$ is a Young function, it is not necessary that both $f$ and $\varphi$ are Young functions in Theorem \ref{Main}.
As an example, consider $\Phi(t)=t^{p}\log(e+t)$ , $1<p<2$.
\end{remark}

There are lots of functions $\Phi$ and $\Psi$ for which this capacitary strong type inequality
\eqref{CI} is valid with functions from the Orlicz-Sobolev space $W_{0}^{1, \Phi}(B^n(0,R))$.
We give concrete examples of functions $\Phi$ and $\Psi$ in     
Example \ref{example:concrete} and more general examples in 
Example \ref{example:varphi} and Corollary \ref{Savu}.
As an application we prove results for limits of capacitary averages, Theorem \ref{Application}.

\section{Preliminaries}
Let us recall that
a continuous, strictly increasing convex function $\Phi :[0,\infty )\to [0,\infty )$ is 
called a Young function, if
\begin{equation*}
\lim_{t\to 0+} \frac{\Phi (t)}{t}=\lim_{t\to\infty} \frac{t}{\Phi(t)}=0.
\end{equation*}
It is well known that
\begin{equation*}
\Phi (t)=\int_0^t\phi (s)\,ds
\end{equation*}
with some non-decreasing, right-continuous function 
$\phi $ on $[0,\infty )$.
The so called
$\Delta _2$-condition, which means that for a Young function $\Phi$ there is a positive constant $C$ such that  
\begin{equation}\label{doubling}
\Phi (2t)\le C\Phi (t) 
\end{equation}
for all  $t\in [0,\infty )\,,$
is a useful property.
The condition \eqref{doubling} is called  the doubling condition if the function $\Phi$ is not a Young function. 

Let $G$ be a domain in $\Rn$, $n\geq 2$, that is open and connected.
Let us write
$\overline{\R}=\R\cup\{\infty\}\cup\{-\infty\}$
and
\begin{equation*}
\mathcal{M}_0(G)
=\big\{u: G\to\overline{\R} \big\vert\quad u  \mbox{ is measurable } \big\}\,.
\end{equation*}
The Orlicz class
\begin{equation*}
L_{\Phi}(G)=\bigg\{u\in \mathcal{M}_0(G) \bigg\vert\int_{G}\Phi (\vert u(x)
\vert )\,dx <\infty\bigg\}
\end{equation*}
is not necessarily a linear space.
Hence, the Orlicz space is defined as
\begin{equation*}
L^{\Phi}(G)=\bigg\{u\in \mathcal{M}_0(G) \big\vert
\int_{G}\Phi (\vert \lambda u(x)\vert )\,dx <\infty \mbox{ with some } \lambda >0 \bigg\}.
\end{equation*}
If $\Phi$ satisfies the $\Delta_2$-condition, then in fact
$L^{\Phi}(G)=L_{\Phi}(G)$.
The Orlicz space equipped with the Luxemburg norm,
\begin{equation*}
\vert\vert u\vert\vert _{L^{\Phi}(G)}=
\inf\biggl\{s \bigg\vert
\int_{G}\Phi \bigg(\big\vert \frac{u(x)}{s}\big\vert\bigg)\,dx \le 1\biggr\}\,,
\end{equation*}
is a Banach space. 

The Orlicz-Sobolev space $W^{1,\Phi}(G)$ is defined as
\begin{equation*}
W^{1,\Phi}(G)=\bigg\{u\in L^{\Phi}(G)\big\vert \mbox{ the first weak derivative }
Du \in L^{\Phi}(G)    \bigg\}.
\end{equation*}
It is a Banach space with the norm
\begin{equation*}
\vert\vert u\vert\vert _{W^{1,\Phi}(G)}=
\vert\vert u\vert\vert _{L^{\Phi}(G)}+\vert\vert Du\vert\vert _{L^{\Phi}(G)}\,.
\end{equation*}
 The space $W_{0}^{1,\Phi}(G)$ is the closure of $C_{0}^{\infty}(G)$ with respect to the norm
$\vert\vert \cdot \vert\vert _{W^{1,\Phi}(G)}$.
The functions from the Orlicz, and Orlicz-Sobolev, spaces are called Orlicz, and Orlicz-Sobolev, functions
respectively.

If a Young function $\Phi$ satisfies the $\Delta _2$-condition, then
the space 
$W^{1,\Phi}(G)$ is separable
and also $C^{\infty}_0(\Rn )$ is dense in $W^{1,\Phi}(\Rn)$.
We refer the reader to \cite[Chapter 8]{AF}, \cite[Chapter 3]{KJF} and \cite{RR} for more information about Orlicz spaces and
to \cite[Chapter 7]{KJF} for information about Orlicz-Sobolev spaces.

Throughout the paper, the letters $C$ and $K$ will denote various constants which may differ from one formula to the next. 
The notation $\sim$ means 'is comparable to', that is the ratio of the two quantities is bounded above and below by finite positive constants.
The balls $B^n(0,R)$, $R>0$, are in $\R^n$, $n\geq 2$.

\vspace{0.5cm}

The following
${\Delta}_{2}^{+}$-condition introduced in \cite[Section 4]{J3} is useful in applications.

\subsection*{The ${\Delta}_{2}^{+}$-condition}
Let $1<p<\infty.$
Let $\varphi$ be a positive, increasing and differentiable function on $[0,\infty )$ such that 
\begin{equation*} 
\varphi(t^{2})\sim \varphi(t) \quad \textrm{on}~(0,\infty)
\end{equation*}
and
\begin{equation*} 
\lim_{t\to\infty} \frac{t{\varphi}'(t)}{\varphi(t)}=0.
\end{equation*}
If there are positive constants $C$ and $K$ such that 
\begin{equation*} 
\frac{t{\varphi}'(t)}{\varphi(t)} \leq C <p  \quad\mbox{ for all } t\in(0,\infty)
\end{equation*}
and  
\begin{equation*} 
{\varphi}'(t) \leq K \quad\mbox{ for all } t\in(0,\infty),
\end{equation*}
and if the function $t\mapsto t^{p} \varphi (t)$ is a Young function 
on $[0,\infty )$ with a fixed $p\in (1,\infty)$,
then the function $\Phi(t)=t^{p} \varphi (t),$ $t\geq 0$, 
satisfies the ${\Delta}_{2}^{+}$-condition.

\begin{example}
The function $\Phi$,
\begin{displaymath}
\Phi(t)=t^{p} \left(\log(C+t)\right)^{\theta} \exp\left( [\log\log(C+t)]^{\gamma}\right),\quad t\geq 0
\end{displaymath}
satisfies
the ${\Delta}_{2}^{+}$-condition
when $p\in(1,\infty)$, $\theta \in[0,p-1]$, $\gamma \in [0,1)$, and
$C\geq e^{e}$ is a positive constant depending on $p$, $\theta$ and $\gamma$ only.
We refer the reader to \cite[Example 18]{J3}.
\end{example}

\section{The ${C}_{\Phi}$-capacity}

We recall the definition of the ${C}_{\Phi}$-capacity when 
$\Phi$ is a Young function with the $\Delta_2$-condition.
\begin{definition} \label{capsu} 
Let $R>0$ be given.  
Let $\Phi$ be a Young function such that $\Phi$ satisfies the $\Delta_{2}$-condition. 
For any set $E$ in $B^n(0,R)$ the ${C}_{\Phi}$-capacity of $E$ is defined by
\begin{equation*}
{C}_{\Phi} (E)=\inf \biggl\{\int {\Phi}(|\nabla u(x)|)dx : u\in C_{0}^{\infty}(B^n(0,R)), ~u(x)\geq 1~ \textrm{when } x\in E \biggr\}.
\end{equation*} 
\end{definition} 
A given property holds for almost every $x_0\in\Rn$ 
( a. e. $x_0\in\Rn$ ) in the ${C}_{\Phi}$-capacity 
sense if it holds outside a set of zero 
${C}_{\Phi}$-capacity; this is written as
${C}_{\Phi}$ -a.e. $x_0\in\Rn$.
The set 
$$\{u\in C_{0}^{\infty}(B^n(0,R)) : u (x)\geq 1, x\in E\}$$ 
can be replaced by a larger set 
$$\{u\in W_{0}^{1,\Phi}(B^{n}(0,R)) : u(x)\geq 1~ \textrm{for}~{C}_{\Phi} -\textrm{a.e. on}~ E\}.$$
That is, if $E$ in $B^{n}(0,R)$ is compact and
\begin{equation*}
\mathcal{K}=\inf \biggl\{\int {\Phi}(|\nabla u(x)|)dx : u\in W_{0}^{1,\Phi}(B^{n}(0,R)), ~u(x)\geq 1~ \textrm{for}~{C}_{\Phi} -\textrm{a.e. on}~ E \biggr\},
\end{equation*}
then $C_{\Phi}(E)=\mathcal{K}$.

For the calculations it is easier to use the Riesz capacity. We recall its definition.
Here, $I_{1}(y)=|y|^{1-n}$ is a Riesz kernel. 
Let $E$ be a set in $B^{n}(0,R)$. If the function $\Phi$ satisfies the ${\Delta}_{2}^{+}$-condition, 
\begin{equation*}
{R}_{\Phi} (E)=\inf \biggl\{\int {\Phi}(f(x))dx : f(x)\geq 0, \textrm{spt} f \subset B^{n}(0,R), ~I_{1}*f(x)\geq 1,~ x\in E \biggr\}\,,
\end{equation*} 
the support of $f$ is written as $\spt f$. 
Other papers that explore or use similar capacities are for example
\cite{AC}, \cite{AHS1}, \cite{CS}, \cite{FMOS}, \cite{J1}, \cite{J2}, \cite{J3}, \cite{M}, and   \cite{MS}.

\begin{lemma}\label{capacity_Riesz}
If the function $\Phi$ satisfies the ${\Delta}_{2}^{+}$-condition,
then
$R_{\Phi}(E)\sim C_{\Phi}(E)$
for all compact sets $E$ in $B^{n}(0,R)$.
\end{lemma}

For the proof we need Lemma \ref{R}.
Let
$R=(R_1\,,\dots\,, R_n)$ 
be a Riesz transformation.
Suppose that there is a constant $K$ depending on $n$ only such that
\begin{equation}\label{Riesz}
R_{j} * u(x)=K \int \frac{x_{j}-y_{j}}{|x-y|^{n+1}} u(y) dy, \quad j=1,\ldots,n,
\end{equation}
with a suitable function $u$.

\begin{lemma} \cite[Theorem 1.4.3]{KK}\label{R}
If the function $\Phi$ satisfies the ${\Delta}_{2}^{+}$-condition, $\int \Phi(u(x))dx <\infty$
with a function $u$ from a suitable space, and
\eqref{Riesz} holds, then there exists a constant $C$ such that
\begin{equation*}
\int \Phi(R_ {j}*u(x))dx\leq C\int \Phi(u(x))dx, \quad j=1,\ldots,n.
\end{equation*}
\end{lemma}

\subsection{Proof for the equivalence of the capacities}

\begin{proof}
Let $u\in C_{0}^{\infty}(B^n(0,R))$, with $u(x)\geq 1$ on $E$. Then
\begin{equation*}
1\leq |u(x)|\leq K C_{\Phi}(E)
\end{equation*}
with some constant $K$.
On the other hand, if $I_ {1}*f(x)\geq 1$ on $E$, then by Lemma \ref{R} 
for a Riesz transformation 
$R=(R_{1},\ldots ,R_{n})$ there is a constant $C$ such that
\begin{equation*}
\int \Phi(|\nabla(I_{1}*f)|)dx=\int \Phi(|\nabla(R*f)|)dx \leq C\int \Phi(f)dx\,.
\end{equation*}
Hence $I_ {1}*f$ is a test function for $C_{\Phi}$. Thus, 
\begin{equation*}
C_{\Phi}(E)\leq CR_{\Phi}(E)
\end{equation*} 
with a constant $C$,
\end{proof}

The equivalence of the capacities in Lemma \ref{capacity_Riesz} can be extended to all Suslin sets $A$
in $B^{n}(0,R)$, since
$C_{\Phi}(A)=\sup_{E\subset A} C_{\Phi}(E)$, where $E$ is compact and $R_{\Phi}(A)=\sup_{E\subset A} R_{\Phi}(E)$, where $E$ is compact.

\section{The capacitary strong-type inequality: proof and examples}

We will prove the capacitary strong type inequality 
\eqref{CSTI} for a large class of functions
$\Phi$ and $\Psi$ with Orlicz-Sobolev functions from $W^{1,\Phi}_0(B^n(0,R))$
as formulated in 
Theorem \ref{Main}.
The proof is an extension of the proof of \cite[Theorem 1.9]{AHS2}.

\subsection{Proof for the capacitary strong type inequality in Theorem \ref{Main}}
\begin{proof}
We may assume that
$u\in {C}_{0}^{\infty}(B^{n}(0,R))$.
Since the function $t \mapsto \Psi (t)$ is a $C^1$-function in $(0, \infty)$, 
the Riemann-Stieltjes integral becomes the Riemann integral.
Since $\Psi$ is increasing in $(0, \infty)$, we obtain 
\begin{eqnarray*}
\int_{0}^{\infty} C_{\Phi}(\{x: u(x)>t\})d\Psi(t) &=& \sum_{k=-\infty}^{\infty} \int_{2^{k}}^{2^{k+1}} C_{\Phi}(\{x: u(x)>t\})d\Psi(t) \\
&&\\
&\leq& \sum_{k=-\infty}^{\infty} C_{\Phi}(\{x: u(x)>2^{k}\}) \int_{2^{k}}^{2^{k+1}} \Psi '(t) dt \\
&& \\
&\leq& \sum_{k=-\infty}^{\infty} C_{\Phi}(\{x: u(x)>2^{k}\})(\Psi(2^{k+1})- \Psi(2^{k})).
\end{eqnarray*}
Let $H:(-\infty, \infty) \to \R$ be the function
\begin{equation*}
H(t) = \left\{ \begin{array}{ll}
0, & \textrm{if $t\leq 1/2$}, \\
2t-1, & \textrm{if $1/2 \leq t\leq 1$}, \\
1, & \textrm{if $t\geq 1$}.
\end{array} \right.
\end{equation*}
Then, $0 \leq H'(t) \leq 2$ for all $t\in(0,\infty)$. 
On the set 
\begin{equation*}
\big\{x:\frac{u(x)}{2^{k}} >1\big\}
\end{equation*}
$H(\frac{u(x)}{2^k})=1$.
Let us write
\begin{equation*}
{E}_k=\big\{x: \frac{1}{2}< \frac{u(x)}{2^{k}} \le 1\big\}.
\end{equation*}
Since the function $H(\frac{u(x)}{2^{k}})$ is a test function for 
the $C_{\Phi}$-capacity of the set
$\big\{x: \frac{u(x)}{2^{k}}>1\big\}$, that is
$C_{\Phi}(\big\{x: \frac{u(x)}{2^{k}}>1\big\})$, and $H'(\frac{u(x)}{ 2^{k}})=0$ outside ${E}_k$, 
and $\Phi(0)=0$,
we obtain
\begin{eqnarray*}
   && \int_{0}^{\infty} C_{\Phi}(\{x: u(x)>t\})d\Psi(t) \\
 &\leq& \sum_{k=-\infty}^{\infty} C_{\Phi}(\{x: u(x)>2^{k}\})(\Psi(2^{k+1})- \Psi(2^{k})) \\
  &\leq& \sum_{k=-\infty}^{\infty} (\Psi(2^{k+1})- \Psi(2^{k}))\int \Phi\left(|\nabla H(u(x)/2^{k})|\right) dx.  
\end{eqnarray*}
Further,
\begin{equation*}
\nabla H(u(x)/2^{k}) =H'(u(x)/2^{k}) \cdot \nabla u(x) /2^{k},
\end{equation*}
and $0\leq H'(t)\leq 2$, and $\Phi$ is increasing. Hence
\begin{eqnarray*}
  && \sum_{k=-\infty}^{\infty} (\Psi(2^{k+1})- \Psi(2^{k}))\int \Phi\left(|\nabla H(u(x)/2^{k})|\right) dx  \\
  &=& \sum_{k=-\infty}^{\infty}(\Psi(2^{k+1})- \Psi(2^{k})) \int_{{E}_k} \Phi \left(H'(u(x)/2^{k}) |\nabla u(x)|/2^{k} \right) dx  \\
  &\leq& \sum_{k=-\infty}^{\infty}(\Psi(2^{k+1})- \Psi(2^{k})) \int_{{E}_k} \Phi \left(|\nabla u(x)|/2^{k-1} \right)\,dx .
\end{eqnarray*} 
It follows from
\begin{equation*}
0\leq \Psi(2^{k+1})- \Psi(2^{k}) \leq \Psi(2^{k+1}),
\end{equation*}
and inequalities (\ref{EP}) and (\ref{EP2}) that
\begin{eqnarray*} 
&& \sum_{k=-\infty}^{\infty} \int_{{E}_k} \Phi \left(|\nabla u(x)|/2^{k-1} \right) dx (\Psi(2^{k+1})- \Psi(2^{k})) \\
&\leq& \sum_{k=-\infty}^{\infty} \int_{{E}_k} \Phi \left(|\nabla u(x)|/2^{k-1} \right) \Psi(2^{k+1}) dx \\
&=& \sum_{k=-\infty}^{\infty} \int_{{E}_k} f(|\nabla u(x)| /2^{k-1})\cdot \varphi(|\nabla u(x)|/2^{k-1}) \cdot f(2^{k+1}) \cdot \psi(2^{k+1}) dx \\
&\leq& C \sum_{k=-\infty}^{\infty} \int_{{E}_k} f(|\nabla u(x)| \cdot 2^{-k+1}\cdot 2^{k+1}) \varphi(|\nabla u(x)| \cdot 2^{-k+1} \cdot 2^{k+1})\,dx .  
\end{eqnarray*}
Since functions $f$ and $\varphi$ satisfy  the doubling condition,  the previous estimates yield
\begin{eqnarray*}
\int_{0}^{\infty} C_{\Phi}(\{x: u(x)>t\})d\Psi(t)
&\leq& C \sum_{k=-\infty}^{\infty} \int_{{E}_k} f(4|\nabla u(x)|) \varphi(4|\nabla u(x)|)  dx  \\
&\leq& C \sum_{k=-\infty}^{\infty} \int_{{E}_k} f(|\nabla u(x)|)\varphi(|\nabla u(x)|) dx \\
&=& C \sum_{k=-\infty}^{\infty} \int_{{E}_k} \Phi(|\nabla u(x)|) dx \\
&=& C  \int \Phi(|\nabla u(x)|) dx . 
\end{eqnarray*}
Hence, the proof is complete.
\end{proof}

Now we give general examples of the functions which satisfy the capacitary strong type inequality
of Theorem \ref{Main}. In the first example the main point is the role of function $\psi$.

\begin{example}\label{example:varphi}
If functions $\varphi$ and $f$ are increasing
on $[0,\infty )$ and the function $t\mapsto f(t) \varphi(t)$ is a Young function on $[0,\infty)$,
and the inequality
\begin{equation*}
\varphi(st) \leq C \varphi(s)\varphi(t)
\end{equation*}
holds with some constant $C$
for all $s,t \in(0,\infty)$, 
then Theorem \ref{Main} is valid
with $\psi(t)= (\varphi(1/t))^{-1}$. 
To verify the claim we have to check the validity of the inequality
\begin{equation*}
\varphi(s) \cdot\varphi\left(\frac{1}{t}\right)^{-1} \leq C \varphi(st)
\end{equation*}
for all $s,t \in(0,\infty)$.
But the inequality holds, since
$\varphi(st) \leq C \varphi(s)\varphi(t)$ for all $s,t \in(0,\infty)$ and
we may calculate
\begin{eqnarray*}
\varphi(s) \cdot\varphi\left(\frac{1}{t}\right)^{-1} &=& \varphi\left(st\cdot \frac{1}{t}\right) \varphi\left(\frac{1}{t}\right)^{-1} \leq C \varphi(st) \cdot \varphi\left(\frac{1}{t}\right) \varphi\left(\frac{1}{t}\right)^{-1} \\
                                                            &=& C \varphi(st).
\end{eqnarray*}
\end{example}

The next result shows that our capacitary strong-type inequality is valid 
for a large class of Young functions $\Phi$, namely the ones which
satisfy the ${\Delta}_{2}^{+}$-condition.
\begin{corollary} \label{Savu}
Let $1<p<\infty$ be fixed.
Let the Young function $\Phi(t)=t^{p} \varphi(t)$ satisfy the 
${\Delta}_{2}^{+}$-condition on $[0,\infty)$ and
let $\Psi(t)=t^{p} \psi(t)$ be an increasing $C^1$-function
on $(0,\infty )$
such that with some constant $C$
\begin{equation*}
\varphi(s) \cdot \psi(t) \leq C \varphi(st) \quad \textrm{for all}~s \geq 0~\textrm{and}~t>0.
\end{equation*}
Then,
if a function $u$ belongs to the Orlicz-Sobolev space $W_{0}^{1, \Phi}(B^n(0,R))$, 
there is a constant $K$, independent of function $u$,
such that
\begin{equation*} 
\int_{0}^{\infty} C_{\Phi} (\{x\in B^n(0,R): |u(x)|>t\})d \Psi(t) \leq K \int \Phi (|\nabla u(x)|) dx.
\end{equation*}
\end{corollary}

Now, we give concrete examples of functions $\Phi$ and $\Psi$ which satisfy the conditions of Theorem \ref{Main}, the functions $f$ and $\psi$ being
$f(t)=t^{p}$ with $p\in (1,\infty)$ and $\psi(t)=(\varphi(\frac{1}{t}))^{-1}$:

\begin{example}\label{example:concrete}
Let functions $f$ and $\psi$ be
$f(t)=t^{p}$ with $p\in (1,\infty)$ and $\psi(t)=(\varphi(1/t))^{-1}$ in Theorem \ref{Main}. If we choose the functions $\Phi$ and $\Psi$ in one of the following ways,
then the capacitary inequality of Theorem \ref{Main} holds.
\begin{align*}
& (1)\quad \mbox{Let } \theta \in [0,\infty)\mbox{ and }\Phi(t)=t^{p} (\log(e+t))^{\theta}
\mbox{ and } \Psi(t)= t^{p} (\log(e+\frac{1}{t}))^{-\theta}\,.\\
& (2)\quad \mbox{Let } \theta \in [0,1)\mbox{ and }\Phi(t)=t^{p} \exp {((\log\log(C+t))^{\theta})}  \\
&\quad\quad\mbox{ and } \Psi(t)= t^{p} \exp \big({-(\log\log(C+\frac{1}{t}))^{\theta}}\big)\,  \mbox{ with some constant } C \\
&\quad\quad
\mbox{ which depends on } p \mbox{ and } \theta \mbox{ only}.\\
& (3)\quad \mbox{Let } \theta \in [0,1)\mbox{ and }\Phi(t)=t^{p} \exp {((\log(e+t))^{\theta})}\\ 
&\quad\quad\mbox{ and }\Psi(t)= t^{p} \exp\big({-(\log(e+\frac{1}{t}))^{\theta}}\big)\,.
\end{align*}
We omit the calculations, which are straightforward.
\end{example}

\section{Capacitary averages}
As an application of our capacitary strong type inequality we show that
certain capacitary averages tend to zero almost everywhere in the given capacity sense.

We need the following estimates for capacities of balls.
For the proof we refer the reader to \cite[Lemma 7.3]{M} where
similar estimates with 
the Riesz capacity
$R_{\Phi}$ are considered.
Since
$R_{\Phi} (B^{n}(0, r))\sim C_{\Phi} (B^{n}(0, r))$, $0<r<R$,
the proof for Lemma \ref{C} follows easily.

\begin{lemma} \label{C}
Let $B^{n}(0,R)$ be a fixed ball, $R>0$, $n\geq 2$.
Suppose that a Young function $\Phi(t)=t^{n} \varphi(t),$
$t\geq 0$, 
satisfies the ${\Delta}_{2}^{+}$-condition.
If $r\in(0, R/2)$ and
\begin{displaymath}
\mathcal{F}(r)=\int_{r}^{R}s^{-1}\varphi\left(\frac{1}{s}\right)^{-\frac{1}{n-1}}ds, 
\end{displaymath}
then there is a positive constant $K$,
depending on $n,R$ and $\varphi$ only, such that 
\begin{displaymath}
K^{-1} \mathcal{F}(r)^{1-n}\leq C_{\Phi} (B^{n}(0, r)) 
\leq K \mathcal{F}(r)^{1-n}\,.
\end{displaymath}
\end{lemma}

We obtain the next theorem on capacitary averages by using our capacitary strong type inequality.
The proof is similar to the proof of \cite[Lemma 6.1]{AHS2}.

\begin{theorem}\label{Application}
Let $n\geq 2$.
Let the Young function
$\Phi(t)=t^{n}\varphi(t),$ $t\geq 0$, 
satisfy the ${\Delta}_{2}^{+}$-condition. 
Let $\Psi(t)=t^{n} \varphi(\frac{1}{t})^{-1}$, $t>0$, be an increasing $C^1$-function.
If
\begin{equation*}
{E}_{t}(r)=B^{n}({x}_{0},r)\cap \{x:|u(x)-u({x}_{0}|>t\},
\end{equation*}
where $0<r<R$ and $u\in W_{0}^{1,\Phi}(B^n(0,R))$,
then
\begin{equation*}
\lim_{r\to 0} \frac{1}{{C}_{\Phi} (B^{n}({x}_{0},r))} \int_{0}^{\infty} {C}_{\Phi}({E}_{t}(r)) d\Psi(t)=0
\end{equation*}
for almost every $x_0\in\Rn$ in the ${C}_{\Phi}$-capacity sense.
\end{theorem}

\begin{proof}
Let us write $F(x)=|u(x)-u(x_{0})|$ and define the maximal operator
\begin{equation*}
\mathcal{M}(F)(x_{0})=\sup_{r>0} \frac{1}{C_{\Phi}(B^{n}(x_{0},r))} \int_{0}^{\infty} C_{\Phi}(E_{t}(r)) d\Psi(t)\,. 
\end{equation*}
In order to show the maximal operator satisfies the weak-type inequality
\begin{equation} \label{WTI}
C_{\Phi}(x\in B^{n}(0,R) : \mathcal{M}(I_{1}*F)(x)>t\}) \leq \frac{C}{\Phi(t)} \int \Phi(F(x))dx
\end{equation}
we divide the function $F$ into two parts:
$F(x)=(F_{1}+F_{2})(x)$ where  $F_{1}(x)=F(x)\chi_{B^{n}(x_{0},2r)}(x)$. The function $\chi_{A}$ is the characteristic function
of a set $A$.
By the capacitary strong-type inequality Theorem \ref{Main} and by Lemma \ref{R} 
there is a constant $C$ such that
\begin{align*}
&\int_{0}^{\infty} C_{\Phi}(\{x\in B^{n}(x_{0},r) :(I_{1}*F_{1})(x)>t\})d\Psi(t) \\
&\leq C \int \Phi(|\nabla I_{1}*F_{1}(x)|)dx \leq C \int \Phi(F_{1}(x))dx \\
&\leq C \int\limits_{B^{n}(x_{0},2r)} \Phi(F(x))dx\,.
\end{align*}
On the other hand, $(I_{1}*F_{2})(x)\leq(I_{1}*F)(x_{0})$ whenever $x\in B^{n}(x_{0},r)$. Thus, by Lemma \ref{C}

\begin{equation} \label{Max}
\mathcal{M}(I_{1}*F)(x_{0})\leq C \sup_{r>0} \frac{1}{C_{\Phi}(B^{n}(x_{0},r))} \int\limits_{B^{n}(x_{0},r)} \Phi(F(x))dx + C(I_{1}*F)(x_{0}).
\end{equation}
Let us write
\begin{equation*}
A_{t}=\biggl\{x : \sup_{r>0} \frac{1}{C_{\Phi}(B^{n}(x,r))} \int\limits_{B^{n}(x,r)} \Phi(F(y))dy>\Phi(t) \biggr\}.
\end{equation*}
For each $x\in A_{t}$ there is $B_{x}$, a ball with center $x$, so that
\begin{equation} \label{Bx}
\int_{B_{x}}\Phi(F(y))dy>\Phi(t)C_{\Phi}(B_{x}).
\end{equation}
By a covering argument, \cite[1.6]{S}, there exists a family of disjoint balls ${B}_{j}$ such that
$A_{t}\subset \bigcup_{j=1}^{\infty} 5B_{j}$. By inequality (\ref{Bx})
\begin{align} \label{CD}
&C_{\Phi}(A_{t}) \leq C \sum_{j=1}^{\infty} C_{\Phi}(B_{j})\notag \\
&\leq \frac{C}{\Phi(t)}\sum_{j=1}^{\infty}\int_{B_{j}} \Phi(F(x))dx \nonumber \\
&\leq \frac{C}{\Phi(t)}\int_{\Rn} \Phi(F(x))dx.
\end{align}
Hence, inequality (\ref{Max}), estimate (\ref{CD}) and the definition of $C_{\Phi}$ imply the weak-type inequality (\ref{WTI}). 
Since 
the function $\Phi$ satisfies the $\Delta _2$-condition,
$C_{0}^{\infty}(\Rn)$ is dense in $W^{1,\Phi}(\Rn)$=$W^{1,\Phi}_0(\Rn)$
by \cite[Theorem 2.1]{DT}. Hence the weak-type inequality gives the claim.
\end{proof}

\end{document}